\documentclass{amsart}

\usepackage{amsmath,amssymb}

\title
[An obstruction for codimension two contact embeddings]
{An obstruction for the existence of codimension two contact embeddings in a Darboux chart}

\author{Naohiko Kasuya}
\address{Graduate School of Mathematical Sciences, University of Tokyo, 3-8-1 Komaba, Meguro-ku, Tokyo 153-8914, Japan.}
\email{nkasuya@ms.u-tokyo.ac.jp}
\date{April 24, 2013.}

\newcommand{\Real}{\mathbb{R}}
\newcommand{\Complex}{\mathbb{C}}
\newcommand{\Integer}{\mathbb{Z}}

\theoremstyle{plain}
\newtheorem{theorem}{Theorem}[section]
\newtheorem{prop}[theorem]{Proposition}

\newtheorem{corollary}[theorem]{Corollary}
\theoremstyle{definition}
\newtheorem{definition}[theorem]{Definition}

\newtheorem{remark}[theorem]{Remark}

\begin{document}

\begin{abstract}
We prove the vanishing of the first Chern class of a codimension $2$ closed contact submanifold of  
a cooriented contact manifold with trivial integral $2$-dimensional cohomology group. 
Hence the first Chern class is an obstruction for the existence of codimension $2$ contact embeddings in a Darboux chart.
For the existence of such an embedding, we prove that a closed cooriented contact $3$-manifold  
can be a contact submanifold of $\Real^5$ for a certain contact structure, if its first Chern class vanishes. 
\end{abstract}

\maketitle

\section{Introduction}
In this paper, we study codimension $2$ contact embeddings in the odd dimensional Euclidean space. 
Let $(M^{2n-1},\xi)$ be a closed contact manifold and $(N^{2m-1},\eta )$ be a cooriented contact manifold. An embedding $f:M^{2n-1}\to N^{2m-1}$ is said to be a contact embedding if $f_{\ast}(TM^{2n-1})\cap \eta |_{f(M^{2n-1})}=f_{\ast}\xi $. 
Note that $\xi $ must be coorientable since $f^{\ast }\beta $ is a global defining $1$-form of $\xi$, where $\beta $ is a global defining $1$-form of $\eta $.
For given $(M^{2n-1},\xi)$, we would like to know whether there exists a contact embedding of $(M^{2n-1},\xi)$ in $(\Real^{2n+1},\eta _{0})$, 
where $\eta _{0}$ is the standard contact structure on $\Real^{2n+1}$. 
It is equivalent to the existence of contact embeddings of $(M^{2n-1},\xi)$ in the $(2n+1)$-sphere with the standard contact structure.
We see that the first Chern class is an obstruction for the existence of such an embedding.

\begin{theorem}~\label{obstruction}
If a closed contact manifold $(M^{2n-1},\xi)$ is a contact submanifold of 
a cooriented contact manifold $(N^{2n+1},\eta )$ such that $H^2(N^{2n+1};\Integer )=0$, 
then the first Chern class $c_1(\xi)$ vanishes.
\end{theorem}

In particular, there are infinitely many contact $3$-manifolds which cannot be embedded in $(\Real^5,\eta_0)$ as contact submanifolds. 
We note that any $3$-manifold can be embedded in $\Real^5$ by Wall's theorem \cite{Wa65}. 
We also note that A.Mori constructed a contact immersion of any coorientable contact $3$-manifold in $(\Real^5,\eta_0)$. 
For the existence of contact embeddings of contact $3$-manifolds in $(\Real^5, \eta _0)$, there are several known examples. 
Some of them are singularity links. 
Let $X$ be a complex surface in $\Complex^{3}$ with an isolated singularity at the origin $0$.
The intersection $L^{3}$ of $X$ and a sufficiently small sphere $S_{\varepsilon }^{5}$ is called the link of $(X,0)$.
The canonical contact structure $\xi$ on $L^3$ is given by $\xi=TL^{3}\cap JTL^{3}$, 
where $J$ is the standard complex structure on $\Complex^{3}$.
It is obviously a contact submanifold of $(S^5,\eta_{std})$, where $\eta_{std}$ is the standard contact structure on $S^5$. 
Though it is difficult to determine the structure on a link in general, 
it is done in the cases of the quasi-homogeneous singularities and the cusp singularities(\cite{Ka12}, \cite{Mo09}, \cite{Ne83}). 
In these cases, the link is the quotient of a cocompact lattice of a Lie group $G$ and the contact structure is invariant under the action of $G$. 
For example, the link of the Brieskorn singularity $\left\{x^p+y^q+z^r=0 \right\}$ 
is a quotient of $G=SU(2),Nil^3$ or $\widetilde{SL}(2;\Real) $, according as the rational number $p^{-1}+q^{-1}+r^{-1}-1$ is positive, zero or negative \cite{Mi75}. 
For the cusp singularity $\left\{x^p+y^q+z^r+xyz=0 \right\}$ ($p^{-1}+q^{-1}+r^{-1}<1$), the link is a Sol-manifold, namely, a quotient of the Lie group $G=Sol^3$ \cite{Ka12}. 
Another example is given by A.Mori \cite{Mo12} and Niederkr\"{u}ger-Presas \cite{NP10}.
They independently constructed a contact embedding of the overtwisted contact structure on $S^3$ 
associated to the negative Hopf band in $(S^5,\eta_{std})$. 
In a similar way as Mori's construction, we can easily see 
that tight contact structures on the $3$-torus also can be embedded in $(S^5,\eta_{std})$ as contact submanifolds. 
In spite of these examples, 
we do not know whether every contact $3$-manifold with $c_1(\xi)=0$ can be embedded as a contact submanifold in $(\Real^5,\eta _0)$. 
By Gromov's h-principle, however, we can show the following result. 

\begin{theorem}~\label{embedding}
We can embed $(M^3,\xi)$ in $\Real^5$ as a contact submanifold for some contact structure 
if $c_1(\xi)=0$.
\end{theorem}

\section{Preliminary}

\subsection{The Chern classes of a cooriented contact structure}
Let $(M^{2n-1},\xi =\ker \alpha )$ be a cooriented contact structure.
Since the $2$-form $d\alpha $ induces a symplectic structure on $\xi $, 
$(\xi ,d\alpha |_{\xi })$ is a symplectic vector bundle over $M^{2n-1}$. 
Since the conformal class of the symplectic bundle structure does not depend on the choice of $\alpha $,   
we define the Chern classes of $\xi$ as the Chern classes of this symplectic vector bundle.

\subsection{The tubular neighborhood theorem for contact submanifolds}
\begin{definition}
Let $(M,\xi )$ and $(N,\eta )$ be cooriented contact structures.
An embedding $f:M\to N$ is said to be a contact embedding if $f_{\ast}(TM)\cap \eta |_{f(M)}=f_{\ast}\xi $. 
The embedded contact manifold $(f(M), f_{\ast } \xi )$ or $(M,\xi )$ itself is called a contact submanifold of $(N,\eta )$.  
\end{definition}

\begin{remark}
The condition $f_{\ast}(TM)\cap \eta |_{f(M)}=f_{\ast}\xi $ is equivalent to $\ker (f^{\ast }\beta) =\ker \alpha $, 
where $\alpha $ and $\beta $ are defining $1$-forms of $\xi $ and $\eta $, respectively.
\end{remark}

Let $(M,\eta _M)\subset (N,\eta =\ker \beta )$ be a contact submanifold. 
The vector bundle $\eta$ splits along $M$ into the Whitney sum of the two subbundles
$$\eta |_M=\eta _M\oplus (\eta _M)^{\perp } ,$$ where $\eta _M$ is the contact plane bundle on $M$ given by $\eta _M=TM\cap \eta |_M$ 
and $(\eta _M)^{\perp }$ is the symplectic orthogonal of $\eta _M$ in $\eta |_M$ with respect to the form $d\beta $. 
We can identify $(\eta _M)^{\perp }$ with the normal bundle $\nu M$. Moreover, $d\beta $ induces a conformal symplectic structure on $(\eta _M)^{\perp }$.
We call $(\eta _M)^{\perp }$ the conformal symplectic normal bundle of $M$ in $N$.

\begin{theorem}[Theorem 2.5.15 of \cite{Ge08}]~\label{tubular nbd}
Let $(N_i,\eta _i)$, $i=1,2$, be contact manifolds with compact contact submanifolds $(M_i,\xi_i)$.
Suppose there is an isomorphism of conformal symplectic normal bundles $\Phi \colon ({\eta_1}_{M_1})^{\perp }\to ({\eta_2}_{M_2})^{\perp }$ 
that covers a contactomorphism $\phi \colon (M_1,\xi_1)\to (M_2,\xi_2)$.
Then there exists a small neighborhood of $M_1$ in $N_1$ that is contactomorphic to a small neighborhood of $M_2$ in $N_2$.
\end{theorem}

\subsection{The Euler class of the normal bundle of an embedding}

\begin{theorem}[Theorem 11.3 of \cite{MS74}]~\label{normal}
Let $K^k$ be a closed orientable $k$-manifold, 
$L^l$ an orientable $l$-manifold such that $H^{l-k}(L^l;\Integer)=0$  
and $f\colon K^k\to L^l$ an embedding. 
Then the Euler class of the normal bundle vanishes.
\end{theorem}
In particular, when $l=k+2$, the normal bundle is a $2$-dimensional trivial vector bundle.

\section{Proof of Theorem~\ref{obstruction}}
\begin{proof}
Let $f:M^{2n-1}\to N^{2n+1}$ be an embedding such that $$f_{\ast }(TM^{2n-1})\cap  \eta |_{f(M^{2n-1})}=f_{\ast}\xi .$$ 
Since $H^2(N^{2n+1};\Integer)=0$ and the normal bundle of $f$ is $2$-dimensional, it is topologically trivial by Theorem~\ref{normal}.  
Since the conformal symplectic structure on $2$-dimensional trivial vector bundle is unique,   
the normal bundle of $f(M^{2n-1})$ is also trivial as a conformal symplectic vector bundle. 
That is, the vector bundle $\eta $ splits along $f(M^{2n-1})$ such that 
$$\eta | _{f(M^{2n-1})}=\eta _{f(M^{2n-1})}\oplus (\eta _{f(M^{2n-1})})^{\bot },$$ 
where $\eta _{f(M^{2n-1})}=f_{\ast }\xi $ and $(\eta _{f(M^{2n-1})})^{\bot }$ is a trivial symplectic bundle. 
By the naturality of the first Chern class and the condition that $H^2(N^{2n+1};\Integer )=0$, 
it follows that $c_1(\eta | _{f(M^{2n-1})})=f^{\ast }c_1(\eta )=0$.
On the other hand, taking the Whitney sum with a trivial symplectic bundle does not change the first Chern class. 
Thus, $c_1(\eta |_{f(M^{2n-1})})=c_1(\xi )$ holds.  
It follws that $c_1(\xi )=0$.
\end{proof}

Eliashberg's theorem about the weak equivalence between homotopy classes of plane fields and those of overtwisted contact structures 
on an oriented closed $3$-manifold \cite{El89} gives us the following corollary of Theorem~\ref{obstruction}.

\begin{corollary}~\label{overtwisted}
Let $\eta $ be a plane filed over $M^3$ with $c_1(\eta)\neq 0$.
Then there exists an overtwisted contact plane field $\xi$ in the same homotopy class as $\eta$ which cannot be a contact submanifold of $(\Real^5,\eta_0)$.
\end{corollary}

\section{Proof of Theorem~\ref{embedding}}

\subsection{h-principle }

We review Gromov's h-principle and prove Propositon~\ref{key prop} 
as a preliminary for the proof of Theorem~\ref{embedding}.

\begin{definition}
Let $N^{2n+1}$ be an oriented manifold.
An almost contact structure on $N^{2n+1}$ is a pair $(\beta_1,\beta_2)$ 
consisting of a global $1$-form $\beta_1$ and a global $2$-form $\beta_2$ 
satisfying the condition $\beta_1\wedge \beta_2^n\not=0$.   
\end{definition}

\begin{remark}
There is another definition. 
We can define an almost contact structure on $N^{2n+1}$ 
as a reduction of the structure group of $TN^{2n+1}$ from $\mathrm{SO}(2n+1)$ to $\mathrm{U}(n)$. 
Since a pair $(\beta_1,\beta_2)$ satisfying $\beta_1\wedge \beta_2^n\not=0$ can be seen as the cooriented hyperplane field 
$\ker \beta_1$ with an almost complex structure compatible with the symplectic structure $\beta_2|_{\ker \beta_1}$, 
the two definitions are equivalent up to homotopy. 
\end{remark}

\begin{theorem}[Gromov\cite{Gr69}, Eliashberg-Mishachev\cite{EM02}]~\label{Gromov}
Let $N^{2n+1}$ be an open manifold.
If there exists an almost contact structure over $N^{2n+1}$, 
then there exists a contact structure on $N^{2n+1}$ 
in the same homotopy class of almost contact structures. 
Moreover if the almost contact structure is already a contact structure   
on a neighborhood of a compact submanifold $M^m \subset N^{2n+1}$ with $m<2n$, 
then we can choose a contact structure on $N^{2n+1}$ which coincides with the original one on a small neighborhood of $M^m$. 
\end{theorem}

In \cite{Gr69}, Gromov states only the former statement. 
The latter follows from the relative version of Holonomic Approximation Theorem of Eliashberg and Mishachev \cite{EM02}.
Let $(M^{2n-1},\xi=\ker {\alpha })$ be a closed cooriented contact manifold and $M^{2n-1}$ be embedded in $\Real^{2n+1}$.
By Theorem~\ref{normal}, there exists an embedding $$F\colon M^{2n-1}\times D^2 \to \Real^{2n+1}.$$
The form $\alpha +r^2d\theta$ induces a contact form $\beta $ on $U=F(M^{2n-1}\times D^2)$.
By Theorem~\ref{Gromov}, in order to extend given contact structure, it is enough to extend it as an almost contact structure. 
Almost contact structures on $N^{2n+1}$ correspond to sections of the principal $\mathrm{SO}(2n+1)\slash \mathrm{U}(n)$ bundle 
associated with the tangent bundle $TN^{2n+1}$.
In particular, almost contact structures on $\Real^{2n+1}$ correspond to smooth maps $$\Real^{2n+1}\to \mathrm{SO}(2n+1)\slash \mathrm{U}(n).$$
Thus we get the following proposition.

\begin{prop}~\label{key prop}
We can embed $(M^{2n-1},\xi)$ in $\Real^{2n+1}$ as a contact submanifold for some contact structure, 
if and only if there exists an embedding 
$$F\colon M^{2n-1}\times D^2\to \Real^{2n+1}$$
such that the map $g:M^{2n-1}\to \mathrm{SO}(2n+1)\slash \mathrm{U}(n)$ 
induced by the underlying almost contact structure of $(M^{2n-1}\times D^2, \alpha +r^2d\theta )$ is contractible.
\end{prop}
\begin{proof}
The underlying almost contact structure of $(U,\beta )\subset \Real^{2n+1}$ is identified with the map 
$\tilde{g} :U\to \mathrm{SO}(2n+1)\slash \mathrm{U}(n)$ whose restriction to $M^{2n-1}$ is $g$. 
We can take an extension of $\tilde{g} $ over $\Real^{2n+1}$ if and only if $g$ is contractible. 
\end{proof}

\subsection{Proof of Theorem~\ref{embedding}}
 
\begin{proof}
There exists an embedding $f\colon M^3\to \Real^5$ \cite{Wa65}, and the normal bundle of $f$ is trivial. 
Thus we can take an embedding $$F\colon M^3\times D^2\to \Real^5 \text{.}$$
By Proposition~\ref{key prop}, it is enough to prove that if $c_1(\xi)=0$, then there exists an embedding $F$ such that 
the map $g\colon M^3\to \mathrm{SO}(5)\slash \mathrm{U}(2)$ induced by $F$ is contractible. 
Let us take a triangulation of $M^{3}$ and $M^{(l)}$ be its $l$ dimensional skeleton, i.e.,
$$M^{(0)}\subset M^{(1)}\subset M^{(2)}\subset M^{(3)}=M^3.$$ 
By Bott's theorem \cite{Bo59},   
\begin{align*}
\pi_k (\mathrm{SO}(2n+1)\slash \mathrm{U}(n))=\pi_{k+1} (\mathrm{SO}(\infty ))=
\begin{cases}
 0 & (k\equiv 1,3,4,5\text{ mod }8)  \text{,}\\
\Integer & (k\equiv 2,6\text{ mod }8) \text{,}\\
\Integer _2 & (k\equiv 0,7\text{ mod }8)
\end{cases}
\end{align*}
holds for $1\leq k\leq 2n$. Thus $\pi_3 (\mathrm{SO}(5)\slash \mathrm{U}(2))= 0$. 
The condition $c_1(\xi)=0$ is equivalent to that $\xi $ is a trivial plane bundle over $M^3$.
Hence a trivialization $\tau $ of $\xi $ and the Reeb vector field $R$ of $\alpha $ give a trivialization of $TM^3$.
This trivialization of $TM^3$ and a trivialization $\nu $ of the normal bundle $\nu M^3$ form a map $$h\colon M^3\to \mathrm{SO}(5).$$
In other words, $h$ is a trivialization of $T\Real^5\mid _{M^3}$ consisting of 
$R$, $\tau $ and $\nu $.
Composing with the projection $$\pi \colon \mathrm{SO}(5)\to \mathrm{SO}(5)\slash \mathrm{U}(2),$$
it induces the map $g=\pi \circ h\colon M^3 \to \mathrm{SO}(5)\slash \mathrm{U}(2)$.
Thus $h$ is a lift of $g$. 
Now we consider whether $h$ is null-homotopic over $M^{(1)}$.
In other words, we consider the difference between the spin structures on $T\Real^5\mid _{M^3}$ induced by $h$ and the constant map $I_5$.
Then the obstruction is the Wu invariant $c(f) \in \Gamma _2 (M^3)$, 
where $\Gamma_2(M^3)=\left\{C\in H^2(M^3; \Integer)\mid 2C=0 \right\}$.
The following explanation of the Wu invariant is due to \cite{SST}.
The Wu invariant is defined for an immersion of the parallelized $3$-manifold with trivial normal bundle.
A normal trivialization $\nu $ of $f$ and the tangent trivialization define a map $\pi_1(M^3)\to \pi_1(\mathrm{SO}(5))$, 
namely an element $\tilde{c} _f$ in $H^1(M^3;\Integer_2)$.
If we change $\nu $ by an element $z\in [M^3,\mathrm{SO}(2)]=H^1(M^3;\Integer)$, then the class $\tilde{c} _f$ changes by $\rho (z)$,
where $\rho $ is the mod $2$ reduction map $H^1(M^3;\Integer)\to H^1(M^3;\Integer_2)$.
Hence the coset of $\tilde{c} _f$ in $H^1(M^3;\Integer _2)\slash \rho (H^1(M^3;\Integer))$ does not depend on $\nu$.
The cokernel of $\rho $ is identified with $\Gamma_2(M^3)$ by the canonical map induce by the Bockstein homomorphism.
Under this identification, the coset of $\tilde{c}_f$ corresponds to the Wu invariant $c(f)\in \Gamma_2(M^3)$.
Now we fix the trivialization of $TM^3$ formed by $\tau $ and $R$. 
By Theorem 3.8 of \cite{SST}, there exists an embedding $f:M^3\to \Real^5$ such that $c(f)=[0]\in H^1(M^3;\Integer _2)\slash \rho (H^1(M^3;\Integer))$.
Moreover, there exists a normal trivialization $\nu $ of $f$ such that $\tilde{c} _f=0\in H^1(M^3;\Integer _2)$.
With the embedding $f$ and the normal trivialization $\nu $, the map $h$ is null-homotopic over $M^{(1)}$.
Since $\pi_2(\mathrm{SO}(5))=0$, it is also null-homotopic over $M^{(2)}$ and so is the map $g=\pi\circ h:M^3\to \mathrm{SO}(5)\slash \mathrm{U}(2)$.
Since $g$ is null-homotopic over $M^{(2)}$ and $\pi_3(\mathrm{SO}(5)\slash \mathrm{U}(2))=0$, it is contractible.
Therefore if $c_1(\xi)=0$, then there exits an embedding $f\colon M^3\to \Real^5$ such that 
the induced map $g\colon M^3\to \mathrm{SO}(5)\slash \mathrm{U}(2)$ is contractible. This completes the proof of Theorem~\ref{embedding}.
\end{proof}

\section*{Acknowledgements}
I am grateful to Atsuhide Mori and my advisor Professor Takashi Tsuboi. 
Atsuhide Mori motivated me to study the problem of contact embeddings. 
Takashi Tusboi encouraged me constantly and also gave me helpful advice.

\end{document}